\documentclass[11pt,  reqno]{amsart}

\usepackage{amsmath,amssymb,amscd,amsthm,amsxtra, esint, bbm}
\usepackage{xcolor}

\usepackage[implicit=true]{hyperref}

\usepackage{cases}

\allowdisplaybreaks[2]

\newtheorem{theorem}{Theorem} [section]

\newtheorem{lemma}[theorem]{Lemma}
\newtheorem{proposition}[theorem]{Proposition}
\newtheorem{remark}[theorem]{Remark}

\newtheorem*{ack}{Acknowledgments}

\let\Re=\undefined\DeclareMathOperator*{\Re}{Re}
\let\Im=\undefined\DeclareMathOperator*{\Im}{Im}

\newcommand{\om}{\omega}
\newcommand{\Om}{\Omega}
\newcommand{\noi}{\noindent}

\newcommand{\R}{\mathbb{R}}
\newcommand{\C}{\mathbb{C}}

\newcommand{\PP}{\mathbb{P}}

\newcommand{\N}{\mathbb{N}}

\newcommand{\F}{\mathcal{F}}

\newcommand{\al}{\alpha}

\newcommand{\nb}{\nabla}
\newcommand{\Dl}{\Delta}
\newcommand{\eps}{\varepsilon}

\newcommand{\G}{\Gamma}
\newcommand{\ld}{\lambda}

\newcommand{\wt}{\widetilde}
\newcommand{\cj}{\overline}

\newcommand{\dt}{\partial_t}

\newcommand{\les}{\lesssim}

\newcommand{\HS}{\textup{HS}}

\newcommand{\jb}[1]
{\langle #1 \rangle}

\newcommand{\ind}{\mathbf 1}

\newcommand{\cE}{\mathcal{E}}

\newcommand{\E}{\mathbb{E}}
\newcommand{\real}{{\textup{real}}}

\numberwithin{equation}{section}
\numberwithin{theorem}{section}

 \DeclareMathAlphabet{\mathpzc}{OT1}{pzc}{m}{it}

\begin{document}

%\baselineskip = 15pt
%\date{\today}

\title[Energy critical SNLS with non-vanishing B.C.]
{Global well-posedness of the 4-d energy-critical stochastic nonlinear Schr\"{o}dinger equations with non-vanishing boundary condition}

\author[ K.~Cheung and  G.~Li]
{ Kelvin Cheung and Guopeng Li}

\address{
Kelvin Cheung\\
Department of Mathematics\\
Heriot-Watt University and 
the Maxwell Institute for the Mathematical Sciences\\
Edinburgh\\ EH14 4AS\\ United Kingdom}

\email{K.K.Cheung-3@sms.ed.ac.uk}

\address{
Guopeng Li\\
School of Mathematics\\
The University of Edinburgh\\
and The Maxwell Institute for the Mathematical Sciences\\
James Clerk Maxwell Building\\
The King's Buildings\\
Peter Guthrie Tait Road\\
Edinburgh\\ 
EH9 3FD\\United Kingdom} 

\email{guopeng.li@ed.ac.uk}


\vspace*{1mm}


\begin{abstract}
We consider the energy-critical stochastic cubic nonlinear Schr\"odinger equation on $ \R^4 $ with additive noise, and with the non-vanishing boundary conditions at spatial infinity.
By viewing this equation as a perturbation to the energy-critical cubic nonlinear Schr\"odinger equation on $ \R^4 $,
we prove global well-posedness in the energy space. Moreover, we establish
unconditional uniqueness of solutions in the energy space.

\end{abstract}

\maketitle

%\tableofcontents

\section{Introduction}
\label{SEC:1}

\subsection{Stochastic nonlinear schr\"odinger equation}
We study the Cauchy problem for the following (defocusing) energy-critical stochastic nonlinear Schr\"{o}dinger equation (SNLS) with an additive noise on $\R^4$:
\begin{equation}
\label{SNLS}
\begin{cases}
i \dt u + \Dl  u = (|u|^{2}-1) u + \phi \xi \\
u|_{t = 0} = u_0,
\end{cases}
\qquad (t, x) \in \R_+\times \R^4,
\end{equation}
with the non-vanishing boundary condition:
\begin{align}\label{Non-VB}
\lim_{|x|\rightarrow\infty} |u(x)|=1.
\end{align}
Here, $u$ is a complex-valued function, $\R_+$ denotes the non-negative interval $[0,\infty)$, $\xi (t,x)$ denotes a space-time white noise on $\R_+ \times \R^4$, 
and $\phi$ is a bounded operator on $L^2(\R^4)$. Our main goal in this paper is to establish global well-posedness of \eqref{SNLS} subject to \eqref{Non-VB}.

Let us first go over some basic background by considering the following deterministic equation:
\begin{equation}
\label{NLS}
\begin{cases}
i \dt u + \Dl  u =  (|u|^{2}-1) u \\
u|_{t = 0} = u_0,
\end{cases}
\qquad (t, x) \in \R_+\times \R^d,
\end{equation}

\noi
with the non-vanishing boundary condition \eqref{Non-VB}. The equation \eqref{NLS} is also known as the \emph{Gross-Pitaevskii equation} in the literature. 
If $u$ is a solution to \eqref{NLS}.
Then, $\wt{u}=e^{-it}u$ is a solution to the following equation:
\begin{align}\label{NLS-usual}
i \dt \wt{u} + \Dl  \wt{u} 
=  |\wt{u}|^2 \wt{u},
\end{align}

\noi
with boundary condition $\lim_{|x|\to\infty}|\wt{u}(x)|=1$. Here, \eqref{NLS-usual} is the usual (defocusing) NLS on $\R^d$ with cubic nonlinearity. In dimension $d=4$, we say \eqref{NLS-usual} is energy-critical, in the sense that the energy ($=$Hamiltonian)
\begin{align*}
E_0(\wt{u})(t)=\frac{1}{2}\int_{\R^4} |\nb \wt{u}|^2dx+\frac{1}{4}\int_{\R^4} |\wt{u}|^4dx
\end{align*}

\noi
is invariant under the scaling
\begin{align*}
\wt u(t,x)\mapsto \wt u^\lambda(t,x)= \lambda^{-1}\wt u(\lambda^{-2}t,
\lambda^{-1}x)
\qquad 
\text{for }
\qquad
\ld>0,
\end{align*}

\noi
which is also a symmetry for the equation \eqref{NLS-usual}.
For this reason, we also refer to SNLS \eqref{SNLS} as \emph{energy-critical}.

Now, the energy for the Gross-Pitaevskii equation \eqref{NLS} is given by the Ginzburg-Landau energy:
\begin{align}
\label{Eng1}
E(u)(t)=\frac{1}{2}\int_{\R^d} |\nb u|^2dx+\frac{1}{4}\int_{\R^d} \big( |u|^2 -1\big)^2dx,
\end{align}
define on the space
\[
\text{Eng}=\{
u\in H^1_{\text{loc}}(\R^d) : \nb u\in L^2({\R^d}), \, |u|^2-1\in L^2(\R^d)
\}.
\]
In spatial dimensions $d=1,2,3$, G\'erard \cite{G1,G2} proved global well-posedness of \eqref{NLS} with condition \eqref{Non-VB} on the energy space, that is, the space of functions $u$ such that $E(u)<\infty.$ He also proved that the energy space in dimensions $d=3,4$ can be expressed as 
\begin{align}\label{E_gp}
\cE(\R^d):=\big\{ u=\al+v : \alpha\in\C,  |\al|=1, v\in\dot{H}^1(\R^d), |v|^2+2\Re(\bar{\al}v)\in{L}^2(\R^d)  \big\}.
\end{align}
More recently, Killip, Oh, Pocovnicu, and Vi\c san \cite{KOPV} studied \eqref{NLS} for $d=4$ (i.e. the energy-critical case). They treat \eqref{NLS} as a perturbation of the energy-critical NLS \eqref{NLS-usual}, and then utilised the perturbative techniques from Tao, Vi\c san, and Zhang \cite{TVZ} together with the conservation of the energy $E(u)$ to establish unconditional global well-posedness in $\cE(\R^4)$. Our paper is inspired by the work in \cite{KOPV}, where we shall employ similar perturbative techniques on the energy-critical SNLS \eqref{SNLS}.

Outside of the energy space $\cE(\R^d)$, global well-posedness of the Gross-Pitaevskii equation was also established by Zhidkov \cite{Z} for $d=1$, and by Gallo \cite{G}, B\'ethuel, and Saut \cite{BS} for $d=2, 3$. In \cite{G, Z}, they considered initial data from what are now termed \emph{Zhidkov spaces}, while in \cite{BS}, the authors instead considered data from $1+H^1(\R^d)$.

\subsection{Main result}
Let us now turn our attention back to SNLS \eqref{SNLS}. We say that $u$ is a solution to \eqref{SNLS} if it satisfies the non-vanishing boundary condition \eqref{Non-VB} and solves the following Duhamel formulation ($ = $
mild formulation):
\begin{align*}
u(t) = S(t) u_0 -i \int_0^t S(t-t') \big((|u|^{2}-1)u\big) (t') dt' -i \int_0^t S(t-t') \phi \xi (dt'),
\end{align*}
where $S(t) := e^{it \Dl}$ denotes the linear Schr\"{o}dinger propagator. The last term is known as the stochastic convolution and we denote it by 
\begin{align}
\Psi(t) := -i \int_0^t S(t-t') \phi \xi (dt'),
\label{sconv1}
\end{align}
see \eqref{defSC} below for a precise definition.
The regularity of $\Psi$ is dictated by the nature of $\phi$. More specifically, if $\phi \in \HS(L^2; H^s)$, namely, a Hilbert-Schmidt\footnote{We recall the definition of Hilbert-Schmidt operators in Section \ref{SUBSEC:StoCo}.}
operator from $L^2(\R^4)$ to $H^s(\R^4)$, then $\Psi \in C(\R_+; H^s(\R^4))$ almost surely; see Lemma \ref{LEM:sccont} below. Since \eqref{SNLS} is energy-critical, we impose that $\phi\in \HS(L^2;H^1)$.

In the case of zero boundary condition, Oh and Okamoto \cite{OO} employed similar techniques as in \cite{KOPV} to prove global well-posedness for stochastic nonlinear Schr\"odinger equations with an additive noise
(and general power-type nonlinearities)
\begin{align}\label{SNLS0}
i \dt{u} + \Dl {u} =  |{u}|^{p-1} {u} +\phi\xi,
\end{align} 
in $H^1(\R^d)$ in the energy-critical cases; i.e. when $3\le d\le 6$, $p=1+\frac{4}{d-2}$ and $\phi\in \HS(L^2;H^1)$. The authors also established global well-posedness in the \emph{mass-critical case}, see \cite{OO} for more details. We also mention the recent paper \cite{CP} where the first author and Pocovnicu proved local well-posedness of the cubic SNLS (i.e. $p=3$ in \eqref{SNLS0}) with critical data and supercritical noise.  For other works on SNLS with zero boundary condition, see for example \cite{BD}. To the best of our knowledge, there has been no previous work on SNLS with non-zero boundary condition at the time of writing.
%on (i) $L^2(\R^d)$ in the mass-critical case $p=1+\frac4d$ with $\phi \in \HS(L^2(\R^d); L^2(\R^d))$, and (ii) $H^1(\R^d)$ in the energy-critical case $p=1+\frac4{d-2}$ with $\phi \in \HS(L^2(\R^d); H^1(\R^d))$.
%studied \eqref{SNLS0} on $\R^d$. in the energy-subcritical setting, where the Strichartz estimate played an important role. We also mention the recent papers \cite{CP, OPW} on local well-posedness of \eqref{SNLS0} with an additive noise rougher than the critical regularities, i.e. $ \phi\in \HS (L^2;H^s) $ with $ s<s_{crit} $.

Our main result in this paper is as follows:

\begin{theorem}[Global well-poseness of SNLS]\label{THM:0}
Let  $\phi \in \HS(L^2(\R^4); H^1(\R^4))$. Then, the SNLS \eqref{SNLS} with the condition \eqref{Non-VB} is globally well-posed in the energy space $ \cE(\R^4)$. In particular, solutions  are unique in the class $\Psi+ C\big(  \R_+;\cE(\R^4)\big)$.
\end{theorem}

\begin{remark}\rm
Theorem \ref{THM:0} implies global well-posedness of SNLS \eqref{SNLS} subject to zero boundary condition in the energy space $ \cE(\R^4) $. Indeed, one can simply use the transformation $ \wt{u}:=u-\al$ to convert one solution to the other.
\end{remark}

Let us elaborate on our method of the proof. In four dimensions, the energy space $\cE(\R^4)$ can be re-expressed as in \eqref{Eng}. Suppose that solution 
$$u(t)=\alpha+v(t) \in\cE(\R^4),$$

\noi
with\footnote{
In particular,
 it is a complex number of modulus~$1$.
 (since this is about a definition of space of functions
depending only on $x$).
Then, by rotating to set $\al = 1$ (without
loss of generality).
See for example [\cite{G1} Section 5].
Furthermore, this is needed for the initial data such that
$ v_0 = u_0 - 1$
belongs to the right space~($H^1_\text{real}~+~i \dot H^1_\text{real}$).  
As for the
equation, one can always assume that $u$ is written as $u = 1 + v$.  The
point is if $v_0 $ does not belong to the right space, it's useless.

}
$\alpha=e^{i\theta}$. We note here $\al$
% as in \eqref{E_gp}
 is independent of time $t$.
By the gauge invariance in
law (if $\al$ depends on time $t$, this is no longer ture)
of the equation \eqref{SNLS} ($u\mapsto e^{-i\theta}u$). 
Therefore, we can assume
 \footnote{
For equation \eqref{NLS}, 
by the gauge invariance of the equation. It is enough to assume that $\al=1$, see [\cite{KOPV} p.2].
}
 $\theta=0$, and hence $\alpha=1$. 
Furthermore, if we take $ v\in \dot{H}^1(\R^4),$ by the Sobolev embedding $ \dot{H}^1(\R^4)\subset {L}^4(\R^4)$ and \eqref{E_gp},  we have $ \Re(v)  \in {L}^2(\R^4)$. Hence, when $ d=4 $, the energy space is given by
\begin{align}
\label{Eng}
\cE(\R^4):=\{ u=1+v :  v\in H^1_{\real}(\R^4)+i\dot{H}^1_{\real}(\R^4)  \},
\end{align}

\noi
where $H^1_{\textup{real}}(\R^4):=H^1(\R^4;\R) $ is the Sobolev space of real-valued functions, and $\dot{H}^1_{\textup{real}}(\R^4) $ is similarly defined. 
We now rewrite the equation \eqref{SNLS}:
 suppose that $ u =1+v^*$ is a solution to \eqref{SNLS}.
Then, $v^*$ satisfies
\begin{align*}
\begin{cases}
i\dt v^* + \Dl v^* = |v^*|^2v^* +2\Re(v^*) v^* + |v^*|^2 +2\Re(v^*)+\phi \xi\\
v^*|_{t=0}:=u_0-1.
\end{cases}
\end{align*}

\noi
In terms of $v^*$, the energy to the deterministic equation ($\phi \equiv 0$) can then be expressed as
\begin{align}\label{energy}
E(u)(t)=E(v^*+1)(t)=\frac{1}{2}\int_{\R^4}|\nb v^*(t)|^2 dx+\frac{1}{4}\int_{\R^4}\big(|v^*(t)|^2+2\Re(v^*(t))\big)^2\,dx,
\end{align}

\noi
where we continue to denote $E(v^*+1)$ by $E(u)$ for simplicity. It is in this form where we shall establish an a priori bound on the energy of solutions to \eqref{SNLS}, as seen in Proposition \ref{Pro:Bound-Ham}. To actually construct a global-in-time solution, we go one step further and subtract $\Psi$ from $v^*$, which is usually refer to  \emph{Da Prato-Debussche trick} in  the stochastic analysis.	
Write $v := u -1- \Psi$, where $v$ satisfies
\begin{align}\label{SNLS2}
\begin{cases}
i \dt v + \Dl  v =|v|^2v+g(v,\Psi) \\
v|_{t = 0} =v_0:= u_0-1.
\end{cases}
\end{align} 

\noi
Here, $g(v,\Psi)$ is defined to be
\begin{align}
\label{pert}
\begin{split}
g(v,\Psi)&:=(|v+1+\Psi|^{2}-1) (v+1+\Psi)-|v|^2v\\
& =2\Re(v)v+2\Re(\Psi)v+2\Re(\bar{v}\Psi)v\\
&\quad+|\Psi|^2v+|v|^2+2\Re(v)+2\Re(\Psi)+2\Re(\cj{v}\Psi)+|\Psi|^2\\
&\quad+\Psi|v|^2+2\Re(v)\Psi+2\Re(\Psi)\Psi+2\Re(\cj{v}\Psi)\Psi+|\Psi|^2\Psi,
\end{split}
\end{align} 

\noi
which can heuristically viewed as
\begin{align*}
\mathcal{O}\bigg(\sum_{j=1}^3(\Psi+v)^j-v^3\bigg).
\end{align*}
Due to the real parts and the conjugate signs play little to no role in our arguments.
The equation \eqref{SNLS2} can be viewed as the energy-critical NLS \eqref{NLS-usual} with the perturbation $g(v,\Psi)$. As seen later on, the regularity properties of $\Psi$ (Section \ref{SUBSEC:StoCo} below) and the a priori bound on the energy will allow us to invoke the perturbation lemma from \cite{TVZ} (Lemma \ref{LEM:pert} below) on $g(v,\Psi)$ iteratively over any finite time interval to construct a solution $v$ to \eqref{SNLS2}. Finally, the unconditional uniqueness of $v$ needs to be proved via a separate argument adapted from \cite{KOPV}.

The rest of the paper is organized as follows.
In Section \ref{SECT:Prelim}, we introduce some
notations, state regularity properties of the stochastic convolution, and present the key perturbation lemma. 
Then, we give a proof of local well-posedness of the perturbed NLS \eqref{PNLS} in 
Section
\ref{SEC:Pertub}. Next, we apply the perturbation lemma to show the global existence of solutions to the deterministic perturbed NLS.
Finally, Section
\ref{SEC:MainTheorem}
is devoted to proving Theorem \ref{THM:0}. This splits into three parts: the establishment of an a priori bound on the energy, the application of the perturbation lemma, and the proof of unconditional uniqueness.

\section{Preliminaries}\label{SECT:Prelim}

In this section, we introduce some notations
and go over preliminary results.

\subsection{Strichartz estimates}
We now recall the Strichartz estimates. Given $ 0< q, r\le \infty$ and a time interval $I\subseteq \R$, we consider the mixed Lebesgue spaces $L^q_tL^r_x(I\times\R^4)$ of space-time functions $u(t,x) $, endowed with the norm 
\begin{align*}
\|u\|_{L^q_tL^{r}_x(I\times\R^4)}=\bigg(  \int_I\bigg( \int_{\R^4} |u(x,t)|^r \,dx  \bigg)^{\frac{q}{r}}  \,dt\bigg)^{\frac{1}{q}}.
\end{align*}

\noi
We use short-hand notations such as
\[L^q_tL_x^r (I\times \R^4)=L^q(I;L^r(\R^4)); \quad
L^{r}_{t,x}(I\times \R^4)=L^r(I;L^r(\R^4)) , \text{ when }  q=r. 
\]

We say that a pair of exponents $ (q,r)$ is admissible if $ \frac{2}{q}+\frac{4}{r}=2 $ with $ 2\leq q,r\leq\infty $. It is convenient to introduce the following norms. Given a space-time slab $ I\times\R^4$, and $j\in\{0,1\}$, we define the $ \dot{S}^j(I) $-norm by
\begin{equation*}
\|u\|_{\dot{S}^j(I)}:=\sup\big\{\|\nb^ju\|_{L^q_tL^r_x(I\times\R^4)}: (q,r) \mbox{ is admissible}\big\}.
\end{equation*}
We use $ \dot{N}^j(I) $ to denote the dual space of $ \dot{S}^j(I).$ More precisely, we define 
\begin{equation*}
\|u\|_{\dot{N}^j(I)}:=\inf\big\{\|\nb^ju\|_{L^{q'}_tL^{r'}_x(I\times\R^4)}: (q,r) \mbox{ is admissible}\big\},
\end{equation*}
where $(q',r')$ denotes the pair of H\"older conjugates of $(q,r)$. We state the Strichartz estimates in terms of these norms;  see \cite{Strichartz, Yajima, GV, KeelTao}. 
Note, we write $ A \les B $ to denote an estimate of the form $ A \leq CB $, for some constant $ C>0. $

\begin{lemma}[Strichartz estimates]
\label{LM:Stri}
Let $ j\in \{0,1\} $. We have the following homogeneous estimate
\begin{align*}
\|S(t)u_0 \|_{\dot{S}^j(I)}\lesssim \|u_0\|_{\dot{H}^{j}(\R^4)}.
\end{align*}

\noi
For an interval $ I=[t_0,t]\subseteq \R $,
we have the inhomogeneous Strichartz estimate
\begin{align*}
\bigg\| \int_{t_0}^{t} S(t-t')F(t')  dt' \bigg\|_{\dot{S}^j(I)}\lesssim \|F\|_{\dot{N}^j(I)}.
\end{align*}
\end{lemma}

We note down some admissible pairs that will be used throughout this paper:
\begin{align*}
(2,4),\qquad \bigg(6,\frac{12}{5}\bigg), \qquad (\infty,2);
\end{align*}
as well as their corresponding dual indices:
\begin{align*}
\bigg(2,\frac43\bigg),\qquad
\bigg(\frac65,\frac{12}{7}\bigg),\qquad
 (1,2).
\end{align*}
Lastly, given a time interval $ I $, we shall define the space $\dot{X}^1(I)$ endowed with the norm
\begin{align}\label{Xspace}
\|u\|_{\dot{X}^1(I)}:=\|\nb u\|_{L^6_t L^\frac{12}{5}_x(I\times\R^4)},
\end{align}
which serves as an auxiliary space on which we establish local well-posedness.

\subsection{On the stochastic convolution}
\label{SUBSEC:StoCo}
In this section, we record some standard properties of the stochastic convolution $\Psi$ defined in \eqref{sconv1}.
First, we will go through the presise definition of stochastic convolution $\Psi$, see for example in \cite{OO, OPW, ORSW}. 
Given two separable Hilbert spaces $H$ and $K$, we denote by $\HS (H;K)$ the space of Hilbert-Schmidt operators $\phi$ from $H$ to $K$, endowed with the norm:
\[
\| \phi \|_{\HS(H;K)} = \bigg( \sum_{n \in \N} \| \phi e_n \|_K^2 \bigg)^{\frac{1}{2}},
\]
where $\{ e_n \}_{n \in \N}$ is an orthonormal basis\footnote{Recall that the definition of the $\HS(H;K)$-norm is independent of the choice of $\{e_n\}_{n\in\N}$.} of $H$.

Let $(\Om, \F, \PP)$ be a probability space endowed with a filtration $\{ \F_t \}_{t \ge 0}$.
Fix an orthonormal basis $\{ e_n \}_{n \in \N}$ of $L^2(\R^d)$.
Let $W$ be the $L^2 (\R^d)$-cylindrical Wiener process given by
\[
W(t,x,\om) := \sum_{n \in \N} \beta_n (t,\om)  e_n (x),
\]

\noi
and  
$\{ \beta_n \}_{n \in \N}$ 
is defined by 
$\beta_n(0) = 0$ and 
$\beta_n(t) = \jb{\xi, \ind_{[0, t]} \cdot e_n}_{ t, x}$.
Here, we first denote $\jb{\cdot, \cdot}_{t, x}$  as
the duality pairing on $\R\times \R^d$.
Then,
$\xi(t,x)$ denotes the space-time white noise, formally satisfies
\[\mathbb{E}[\xi(t,x)\xi(s,y) ]=\delta(t-s)\delta(x-y).
\]

\noi
Rigorously, we write equation \eqref{SNLS} in the Ito formulation form:
\[
du+(\Dl u)dt=\{(|u|^2 -1 )u\}dt +\phi dW
\]

\noi
Hence, formally $``\xi(t,x)dt=dW(t)"$.
Thus, we formally have
\[
\beta_n(t) = \jb{\xi, \ind_{[0, t]} e_n}_{t, x}
 = \text{``}\int_0^t \int_{\R^d}
\cj{e_n(x)} \xi( dx dt')\text{''}.\]

\noi
As a result, we see that
$\{ \beta_n \}_{n \in \N}$ is a family of mutually independent complex-valued Brownian motions associated to the filtration $\{ \F_t \}_{t \ge 0}$. Let $\phi \in \HS(L^2;H^1)$, and we use the notation
\begin{align*}
\phi_n:=\phi e_n.
\end{align*}

\noi
Then, the space-time white noise $\xi$ is given by a distributional derivative (in time) of $W$ and thus we can express the stochastic convolution $\Psi$ as
\begin{align}
\label{defSC}
\Psi (t)&=-i \int_0^t S(t-t') \phi \xi (dt')    \notag\\
&:=-i\int_0^tS(t-t')\phi\,dW(t') = -i \sum_{n \in \N} \int_0^t S(t-t') \phi_n d \beta_n (t').
\end{align}
Note that the above definition is independent of the choice of the orthonormal basis. The next lemma tells us that $\Psi$ is continuous in time and satisfies a so-called ``Strichartz estimate". The result appeared implicitly in de Bouard-Debussche \cite{BD}, though we borrowed the precise statement from \cite{OPW} where the reader can find a detailed proof for~(i) in \cite{DZ} 
and~\cite{BD, OPW} for (ii).  We remark that this lemma holds for any spatial dimension $d\ge 1$.

\begin{lemma} \label{LEM:sccont}
Let $d\geq 1$, $T>0$, and $ s \in \R $.
Suppose that $ \phi \in \HS(L^2(\R^d);H^s(\R^d)) $. The following properties hold:
\begin{itemize}
\item[(i)]  $ \Psi \in C\big( [0,T];H^s(\R^d)\big)$ almost surely.
Moreover, for any finite $ p \geq 1 $,
there exists $ C=C(T,p)>0 $ such that 
 \[
\E \left[ \sup_{0 \le t \le T} \| \Psi (t) \|_{H^s(\R^d)}^{p} \right] \le C \| \phi \|_{\HS(L^2; H^s)}^{p}.
 \]
 \item[(ii)]  Given any $ 1 \leq q <\infty $ and finite $ r\geq 2 $ such that
$r\leq \frac{2d}{d-2}$ when $d\geq 3$, 
we have $ \Psi \in L^q([0,T];W^{s,r}(\R^d))$ almost surely.
 In particular, for any finite $p \geq 1$,
there exists $C = C(T,p) >0$ such that
 \[
\E \bigg[  \| \Psi \|_{L^q([0,T]; W^{s,r}(\R^d))}^{p} \bigg] \le C \| \phi \|_{\HS(L^2;H^s)}^{p}.
 \]
 \end{itemize}   

\end{lemma}

\subsection{Perturbation lemma}\label{SUBSEC:Perb} Consider the defocusing energy critical NLS equation
\begin{align}\label{NLS1}
i\dt{w}+\Dl{w}=|{w}|^{\frac{4}{d-2}}{w}.
\end{align}
Global well-posedness and scattering for \eqref{NLS1} was proved by Colliander, Keel, Staffilani, Takaoka, and Tao \cite{CKSTT} for spatial dimension $d=3$. Later by Ryckman and Vi\c{s}an \cite{RV} and Vi\c{s}an \cite{V2, V} for $d\ge 4$. An important consequence of these works is that their constructed solutions satisfy a global space-time bounds in Strichartz norms. Specifically, if $ w $ is a
solution to the energy-critical NLS \eqref{NLS1} with initial data $ w_0 \in \dot{H}^1(\R^d)$.
Then, the following global-in-time bound holds:
\begin{align}\label{global-bound}
\|w\|_{\dot{S}^1(\R)}\leq C(\|w_0\|_{\dot{H}^1(\R^d)}).
\end{align}
In \cite{KOPV}, Killip, Oh, Pocovnicu, and Vi\c san  proved global well-posedness of the Gross-Pitaevskii equation \eqref{NLS} by utilising the space-time bounds \eqref{global-bound} in conjunction with a perturbation lemma on \eqref{NLS1}. We shall follow the same basic principles in their work and view \eqref{SNLS} as a energy-critical NLS \eqref{NLS1} with a perturbation on $ \R^4 $.
The key perturbation lemma used in \cite{KOPV} came from \cite[Theorem 3.8]{KV}, and we state this below:

\begin{lemma}[Perturbation lemma]\label{LEM:pert}
Suppose $w_0\in \dot{H}^1(\R^4)$, $ I $ be a compact time interval with $ |I| \leq 1$. Let $ \tilde{w}$ be a solution on $ I\times\R^4$ to the perturbed equation:
\begin{align*}
i\dt\wt{w}+\Dl\wt{w}=|\wt{w}|^{2}\wt{w}+e
\end{align*}
for some function $e.$ There exist functions $\eps_0(E_0,E',L)$ and $\bar{C}(E_0,E',L)$ mapping from $\R_+^3$ to $\R_+$, that are non-increasing in each argument, such that if
\begin{align}
\|\wt{w}\|_{L^6_{t,x}(I\times\R^4)}&\leq L,\label{P1}\\
\|\wt{w}\|_{L^\infty_t\dot{H}^1_x(I\times\R^4)}&\leq E_0,\label{P2}\\
\|\wt{w}(t_0)-w_0\|_{\dot{H}^1(\R^4)}&\leq E',\label{P3}
\end{align} 
for some $t_0\in I$ and positive quantities $L, E_0, E'$, and that
\begin{align}
\|S(t-t_0)(\wt{w}(t_0)-w_0)\|_{\dot{X}^1(I)}&\leq \eps,\label{P4}\\
\|\nb e\|_{\dot{N}^0(I)}&\leq \eps,\label{P5}
\end{align}
for some $ 0<\eps<\eps_0$. Then, there exists a solution $w$ to \eqref{NLS1} with initial data $ w_0 $ satisfying
\begin{align*}
\|w-\wt{w}\|_{L^6_{t,x}(I\times\R^4)}&\leq \bar{C}(E_0,E',L)\eps,\\
\|w-\wt{w}\|_{\dot{S}^1(I)}&\leq \bar{C}(E_0,E',L)E'\\
\|w\|_{\dot{S}^1(I)}&\leq \bar{C}(E_0,E',L).
\end{align*} 
\end{lemma}
\begin{remark}\label{Re:pert}\rm
By the Strichartz estimate, condition \eqref{P4} is redundant if $ E'=O(\eps). $
\end{remark}

\section{Energy-critical NLS with a perturbation}
\label{SEC:Pertub}

In this section, we consider the following defocusing energy-critical NLS with a perturbation:
\begin{align}\label{PNLS}
\begin{cases}
i\dt v+\Dl v=\mathcal N(v+f+1)\\
v|_{t=0}=v_0,
\end{cases}
\end{align}

\noi
for $\mathcal N(v+f+1)=(|v+f+1|^{2}-1)(v+f+1)$.
Here, $f$ is a given deterministic function and
satisfying certain regularity conditions.
By applying the perturbation lemma (Lemma \ref{LEM:pert}), 
we prove global existence for \eqref{PNLS}, 
assuming an a priori energy bound of a solution $v$ to \eqref{PNLS}.
See Proposition \ref{Pro:perturb2}.
In Section~\ref{SEC:global-ex},
we then present the proof of  
Theorem \ref{THM:0} by writing \eqref{SNLS}
in the form \eqref{PNLS}
(with $f = \Psi$) 
and verifying 
the hypotheses in Proposition \ref{Pro:perturb2}.

\subsection{Local well-posedness of the perturbed NLS}
By a standard application of the contraction mapping theorem, 
we have the following local well-posedness of 
the perturbed NLS \eqref{PNLS}. 
%Previously, G\'erard \cite{G1} constructed global solutions to Gross-Pitaevskii equation \eqref{SNLS2} on $\R^4$ for initial data in the energy space \eqref{Eng} with \textbf{small} energy. In fact, his argument is based on contraction mapping arguments in Strichartz spaces.
There is a similar argument by G\'erard [\cite{G1} Theorem 5.1] can be adapted into the following local well-posedness argument.
%More precisely,  G\'erard showed that ([\cite{G1} Proposition 2.3]) the energy space is kept invariant by the Schr\"odinger group $S(t)$. 
We keep the following proof as for reader's convenience.

\begin{proposition}[Local well-posedness of the perturbed NLS]
\label{PRO:PLWP}
Let $I_0$ to be the interval such that $I_0:=[t_0,t_0+T]\subseteq[0,\infty)$. Suppose that 
\begin{align}
\label{a1}
\begin{split}
\|\Re v_0\|_{L^2(\R^4)}+ \|v_0\|_{\dot{H}^1(\R^4)}&\leq R \\
\|\Re f\|_{L^\infty_{t}L^2_x(I_0\times \R^4)}+  
\| f \|_{L^{\infty}_{t}\dot{H}^1_x(I_0\times\R^4)}&\leq M,
\end{split}
\end{align}

\noi
for some $ R,M\geq1 .$
Then, there exists 
%$\dl>0$
%such that for every $1+v_0\in $
some small $ \eta_0=\eta_0(R,M)>0 $ and a compact interval $I\subseteq I_0$ containing $t_0$ such that if 
\begin{align*}
\|S(t-t_0)v_0\|_{\dot{X}^1(I)}+\|f\|_{\dot{X}^1(I)}\leq\eta,
\end{align*}

\noi
for some $ \eta \leq \eta_0 $.
Then, there exists a solution $ v\in C\big(I;H_{\textup{real}}^1(\R^4)+i \dot{H}_{\textup{real}}^1(\R^4)  \big)\cap\dot{X}^1(I) $ to \eqref{PNLS} with $ v(t_0)=v_0. $ Moreover, $v$ satisfies
\begin{align}
\label{LWP1}
\|v-S(t-t_0)v_0\|_{\dot{X}^1(I)}&\le\eta
\end{align}
\end{proposition}

\begin{remark}\rm\label{RM:LWP}
As a consequence of Proposition \ref{PRO:PLWP}. We can prove local well-posedness for the SNLS \eqref{SNLS2} in the space $C (I;\cE( \R^4)) \cap \dot{X}^1(I)  $ (where $ \cE( \R^4)), \dot{X}^1(I) $ are defined in \eqref{Eng} and \eqref{Xspace}).
\end{remark}

\begin{proof}
	
Firstly, we show that the map $\G$ defined by
\begin{equation*}
\G (v(t)) := S(t-t_0) v_0 -i  \int_{t_0}^t S(t -t') \mathcal N(v+f+1) (t') dt',
%\label{LWP2b}
\end{equation*}
	
\noi 
is a contraction on
\begin{equation*}
\label{a2}
\begin{split}
B_{\wt R, \eta} = \Big\{& v \in \dot{X}^{1} (I) 
\cap C\big(I; \dot{H}^1(\R^4)  \big)
:
%\\
%&
%\|\Re v\|_{L^\infty_{t}L^2_x(I\times \R^4)}+
\|v\|_{L^{\infty}_t \dot{H}_x^1(I\times\R^4)}\leq 2\wt R,
\ \|v\|_{\dot{X}^1(I )} \leq 2\eta
\Big\}.
\end{split}
\end{equation*}

\noi
Here, $\wt R :=\max\{R, M\}$.
Let $v_1,v_2,v_3\in B_{\wt R, \eta}$. 
Then, 
by Strichartz, H\"older, and Sobolev inequalities yeild
\begin{align}
\label{REF:LWP1}
\begin{split}
\bigg\| &\int_{0}^{t} S(t-t')\big(v_1\cj{v}_2v_3+v_1\cj{v}_2+v_1\big)(t')dt' \bigg\|_{\dot{S}^1(I)}\\
&\les \sum_{\{i,j,k\}=\{1,2,3\}}\|v_iv_j\nb v_k\|_{L^2_tL_x^{\frac{4}{3}}(I\times \R^4)} 
+\sum_{\{i,j\}=\{1,2\}}\|v_i\nb v_j\|_{L^{\frac{6}{5}}_tL_x^{\frac{12}{7}}(I\times \R^4)}\\
&\quad\quad+ \|\nb v_1\|_{L^{1}_tL_x^{2}(I\times \R^4)}\\
&\les \|v_1 \|_{\dot{X}^1(I)}\|v_2 \|_{\dot{X}^1(I)}\|v_3 \|_{\dot{X}^1(I)}
+ |I|^\frac12\|v_1 \|_{\dot{X}^1(I)}\|v_2 \|_{\dot{X}^1(I)}\\
&\quad\quad+|I|\|v_1 \|_{L^{\infty}_t\dot{H}^1_x(I\times \R^4)},
\end{split}
\end{align}
	%\begin{align}\label{v^3-estimate}
	%\bigg\| \int_{0}^{t} S(t-t')v^3(t')dt' \bigg\|_{\dot{X}^1(I)}&\leq \|v^2\nb v\|_{L^2_tL_x^{\frac{4}{3}}(I\times \R^4)}\notag\\
	%&\leq\|\nb v\|_{L^6_tL_x^{\frac{12}{5}}(I\times \R^4)}\|v\|^2_{L^6_{t,x}(I\times \R^4)}\notag\\
	%&\leq \|\nb v \|^3_{L^6_tL_x^{\frac{12}{5}}(I\times \R^4)}\notag\\
	%&=\|v\|^3_{\dot{X}^1(I)}.
	%\end{align}
	%Next, by Strichartz estimate, H\"older in time, and Sobolev embedding $ W^{1,\frac{12}{5}}(\R^4)\subset L^6(\R^4) $, we have
	%\begin{align}\label{v^2-estimate}
	%\bigg\| \int_{0}^{t} S(t-t')v^2(t')dt' \bigg\|_{\dot{X}^1(I)}
	%&\lesssim \|v\nb v\|_{L^{\frac{6}{5}}_tL_x^{\frac{12}{7}}(I\times \R^4)}\notag\\
	%&\lesssim\ |I|^{\frac{1}{2}}\ | v\|_{L^6_{t,x}(I\times \R^4)}\|\nb v\|_{L^6_{t}L_x^{\frac{12}{5}}(I\times \R^4)}\notag\\
	%&=|I|^{\frac{1}{2}}\|v\|^2_{\dot{X}^1(I)}.
	%\end{align}
	
	%Finally,  by Strichartz estimate, H\"older in time, we have
	%\begin{align}\label{v-estimate}
	%\bigg\| \int_{0}^{t} S(t-t')v(t')  dt' \bigg\|_{\dot{X}^1(I)}
	%&\leq \|\nb v\|_{L^{1}_tL_x^{2}(I\times \R^4)}\notag\\
	%&\lesssim\ |I|\cdot \|\nb v\|_{L^{\infty}_{t}L^2_x(I\times \R^4)}\notag\\
	%&=|I|\cdot\|v\|_{L^{\infty}_t\dot{H}^1_x(I\times \R^4)}.
	%\end{align}

\noi
where $\dot S^1(I)$ denotes the Strichartz spaces with admissible pairs $(6, \tfrac{12}{5})$ and $(\infty,2)$.
Now, in the view of \eqref{pert} the nonlinearility $\mathcal N(v+f+1)$ can be expressed as
\begin{align}
\label{REF:LWP2}
\begin{split}
(|v+f+1|^2-1)(v+f+1)
&=|v|^2v+|f|^2v+2\Re{(\cj{v}f)}v+2\Re{(f+v)}v \\
&\phantom{=}+|v|^2f+|f|^2f+2\Re{(\cj{v}f)}f+2\Re{(f+v)}f \\
&\phantom{=}+|v|^2+|f|^2+2\Re{(\cj{v}f)}+2\Re{(f+v)}.
\end{split}
\end{align}

\noi
We
choose $\eta_0 \ll \wt R^{-1} \leq 1$ and $ |I|\le \min \{1, \eta^3 \wt{R}^{-1} \} $ in the following.
Then,  it follows from \eqref{REF:LWP1} and \eqref{REF:LWP2}  that
there exists small $\eta_0 > 0$ such that
\begin{align*}
%\label{LWP1b}
\|\G (v)\|_{\dot X^1(I)}
&  \leq
\|S(t-t_0)v_0\|_{\dot X^1(I)}
+  \|\mathcal N(v+f+1)\|_{\dot N^1(I)} \notag \\
& \leq \|S(t-t_0)v_0\|_{\dot{X}^1(I)} +  C\bigg( \| v \|^3_{\dot{X}^1(I)}+|I|^{\frac{1}{2}}\| v \|^2_{\dot{X}^1(I)} +|I|\cdot\| v \|_{L^{\infty}_t\dot{H}^1_x(I\times \R^4)}   \notag\\
&\quad+\| f \|^3_{\dot{X}^1(I)}+|I|^{\frac{1}{2}}\| f \|^2_{\dot{X}^1(I)} +|I|\cdot\| f \|_{L^{\infty}_t\dot{H}^1_x(I\times \R^4)}\bigg) \notag \\
&\leq \eta + 2C(2\eta^3 +\eta^{\frac72}   )
%\leq \eta + C \eta^{3}
\leq 2\eta,
\end{align*}

\noi
for $v\in B_{\wt R,\eta}$ and 
 $\eta \leq \eta_0$ sufficiently small.
Similarly, we have
\begin{align*}
\|\G (v)\|_{L^{\infty}_t\dot{H}^1_x(I\times \R^4)}
& \leq  R + C \eta^{3}
\leq 2 \wt R.
%\label{a3}
\end{align*}

\noi
Hence, $ \G $ maps $B_{\wt R, \eta}$ to $B_{\wt R, \eta}$. 
Moreover, the difference estimate follows analogously.
Take $ v_1, v_2 \in B_{ \wt R,\eta}$, we have
\begin{align*}
 \| \G (v_1) - \G (v_2) \|_{L^{\infty}_t\dot{H}^1_x(I\times \R^4)\cap \dot X^1(I)} 
%&\leq C\| \mathcal N(v_1+f+1)-\mathcal N(v_2+f+1)\|_{\dot N^1(I)}\\
%	& \leq C \bigg( \| |v_1^2|v_1-|v_2|^2v_2 \|_{\dot{N}^1(I)}+\|g(v_1,f)-g(v_2,f)\|_{\dot{N}^1(I)}\bigg)  \\
%	& \leq C \bigg( \| v_1 \|_{\dot{X}^1(I)} + \| v_2 \|_{\dot{X}^1(I)} + \| f \|_{\dot{X}^1(I)} \bigg)^2 \times \| v_1-v_2 \|_{\dot{X}^1(I)} \notag \\
%	&\quad + C |I|^{\frac12}\bigg( \| v_1 \|_{\dot{X}^1(I)} + \| v_2 \|_{\dot{X}^1(I)} + \| f \|_{\dot{X}^1(I)} \bigg)  \times \| v_1-v_2 \|_{\dot{X}^1(I)} \notag \\
%	&\quad +C|I|\bigg(  \|v_1-v_2\|_{L^{\infty}_t\dot{H}^1_x(I\times \R^4)} \bigg)\\
	%&\leq C (\eta^2+\tfrac {|I|^{\frac12}}{2} \eta) \| v_1-v_2 \|_{\dot{X}^1(I)}+C|I| \| v_1-v_2 \|_{L^{\infty}_t\dot{H}^1_x(I\times \R^4)}\notag \\
& \leq \frac 12
\| v_1 - v_2 \|_{L^{\infty}_t\dot{H}^1_x(I\times \R^4)\cap \dot X^1(I)} .
\end{align*}

\noi
Therefore, $\G$ is a contraction on $B_{ \wt R, \eta}$. 
Now, we show that solution $v$ we contructed above belongs in to a smaller space:
 \[v\in C\big(I;H_{\textup{real}}^1(\R^4)+i \dot{H}_{\textup{real}}^1(\R^4)  \big).\]

\noi
By using the energy \eqref{Eng1} at initial time (suppose $t_0=0$), H\"older's, and Sobolev inequality, we observe the following:
\begin{align*}
E(v+f+1)(0)&=\frac{1}{2}\int_{\R^4} |\nb (v_0+f)|^2dx+\frac{1}{4}\int_{\R^4} ( |v_0+f|^2+2\Re(v_0+f)   )^2dx\\
%&\leq \frac12( \|v \|^2_{\dot H^1}+ \|f\|^2_{\dot H^1})+\frac14 (\|v\|_{L^4}^4 + \|f\|_{L^4}^4 )\\
%&\quad +\frac12 (\| |v|^2 \|_{L^2}^2+\| |f|^2 \|_{L^2}^2            ) \\
%&\quad +2\{\|\Re v\|_{L^2} (\|\Im f\|_{\dot H^1}\|\Im v\|_{\dot H^1} +\|\Re f\|_{\dot H^1}\|\Re v\|_{\dot H^1} )   \}\\
%&\quad+ (\|v\|_{\dot H^1}^2+ \|f\|_{\dot H^1}^2)( \|\Re f\|_{\dot H^1}\|\Re v\|_{\dot H^1} +\|\Im f\|_{\dot H^1}\|\Im v\|_{\dot H^1}   ) \\
%&\quad+ \|\Re(f\cj v)\|_{L^2} +\|\Re v\|_{L^2}+\|\Re f\|_{L^2}\\
& \leq C( \|\Re v_0\|^2_{L^2}+\|\Re f\|^2_{L^2}+ \|v_0\|_{\dot H^1}^4 + \|f\|_{\dot H^1}^4 +\|v_0 \|^2_{\dot H^1}+ \\
&\quad +
\|f\|^2_{\dot H^1} +\|\Re v_0\|_{L^2}  \|f\|_{\dot H^1} \|v_0 \|_{\dot H^1} +\|f\|_{\dot H^1} \|v_0\|_{\dot H^1}\\
&\quad + \|f\|_{\dot H^1}^3 \|v_0 \|_{\dot H^1}  +\|f\|_{\dot H^1} \|v_0 \|_{\dot H^1}^3  )
\end{align*}

\noi
for some canstant $C>0$.
Therefore, by
the assumption \eqref{a1} and
 for some $\theta>0$, we have the following:
\begin{align}
\label{eng0}
E(v+f+1)(0)\leq C(R+M)^{\theta}<\infty.
\end{align}

\noi
On the other hand, by 
the conservation of energy we observe the following:
\begin{align*}
%\|\Re v\|_{L^\infty_tL_x^2}^2\leq 
E(v+f+1)(0)&=
E(v+f+1)(t)\\
&\geq 
\frac{1}{4}\int_{\R^4}  ( |v+f|^2+2\Re(v+f)   )^2dx\\
%&\geq \frac{1}{4}\int_{\R^4}  |v+f|^2\Re(v+f)  +(\Re(v+f) )^2 dx\\
&\geq \frac{1}{4}\int_{\R^4} (\Re v)^2 + 2\Re v \Re f + \Re(v+f) |v+f|^2 dx.
\end{align*}

\noi
Hence, we rearrange above inequality yields:
\begin{align*}
\|\Re v\|_{L^2(\R^4)}^2\leq 
E(v+f+1)(0)-
\frac{1}{4}\int_{\R^4}   2\Re v \Re f + \Re(v+f) |v+f|^2 dx.
\end{align*}

\noi
 Now, by Young's inequality for every $\eps > 0$, there exists some large constant $C(\eps)=C_{\eps}$ such that
\begin{align*}
\|\Re v\|_{L^2(\R^4)}^2\leq 
E(v+f+1)(0)-
\big( \eps \|\Re v\|_{L^2(\R^4)}^2
+C_{\eps} \|\Re f\|_{L^2(\R^4)}^2\\
+\eps \|\Re v\|_{L^2(\R^4)}^2
+C_{\eps} \|v+ f\|_{L^4(\R^4)}^4 
+\eps \|\Re f\|_{L^2(\R^4)}^2 \big).
\end{align*}

\noi
By Sobolev inequality and rearranging above inequality we obtain:
\begin{align*}
(1-2\eps)\|\Re v\|_{L^2(\R^4)}^2\leq 
E(v+f+1)(0)-
\big(
C_{\eps} \|\Re f\|_{L^2(\R^4)}^2\\
+C_{\eps} \|v+ f\|_{\dot H^1(\R^4)}^4 
+\eps \|\Re f\|_{L^2(\R^4)}^2 \big).
\end{align*}

\noi
Finally, by taking the supreme in time, \eqref{eng0}, and \eqref{a1}, we conclude:
\begin{align*}
\|\Re v\|_{L^{\infty}_tL^2_x(I\times \R^4)}^2&\leq 
(1-2\eps)^{-1}
\big(C(R+M)^{\theta}-
C_{\eps} M^2 
-C_{\eps}\wt R^4 
-\eps M^2 \big)\\
&<\infty
\end{align*}

\noi

\noi
This shows that for the solution $v$ we constructed:
\[
v\in C\big(I;H_{\textup{real}}^1(\R^4)+i \dot{H}_{\textup{real}}^1(\R^4)  \big).
\]
The estimate \eqref{LWP1} is a direct consequence of the above estimates.
\end{proof}

\subsection{Global existence of solutions to the perturbed NLS }\label{SEC:global-ex}
In this subsection, we prove the long time existence of solutions to the perturbed NLS \eqref{PNLS}. 
Given $ T>0 $, we assume that there exist $ C, \theta>0 $ such that
\begin{align}
\label{condition}
%\|f\|_{L^\infty_t\dot{H}^1_x (I\times \R^4)}+
\|f\|_{ \dot{X}^1(I)}\le C |I|^\theta;
\end{align}
for any interval $ I\subset [0,T]$. Then, Lemma \ref{PRO:PLWP} guarantees existence of a solution to the perturbed NLS \eqref{PNLS}, at least for a short time. The following proposition prove the long time existence under some hypotheses.

\begin{proposition}\label{Pro:perturb2}
Let $T>0$ be given, assume the following conditions hold:

\smallskip
\noi
\textup{(i)} 
Let $ f\in  C\big(I;H_{\textup{real}}^1(\R^4)+i \dot{H}_{\textup{real}}^1(\R^4)  \big)\cap\dot{X}^1(I) $ satisfy \eqref{condition}.

\smallskip
\noi
\textup{(ii)} 
Given a solution $v$ to \eqref{PNLS}, we have the following a priori bound
\begin{align}\label{condition2}
\|v\|_{L^\infty_t\dot{H}^1_x([0,T]\times\R^4)}\le R.
\end{align}
\noi
Then, there exists a time $\tau=\tau(R, \theta)>0$ such that given any $t_0\in [0,T)$, a solution $v$ of \eqref{PNLS} exists on $[t_0, t_0+\tau]\cap [0,T]$. This implies that $v$ in fact exists in the entire interval $[0,T]$, as $t_0$ is arbitrary.
\end{proposition}

\begin{proof}
Let $v$ be the local solution to the NLS \eqref{PNLS} (obtained from Proposition \ref{PRO:PLWP}). The main idea is to view \eqref{PNLS} as a perturbation to the energy-critical cubic NLS \eqref{NLS1}on $ \R^4 $, that is, regard $ v $ as $ \wt w $ in Lemma \ref{LEM:pert} with
\[
e=g(v,f),
\]
where $ g(v,f) $ as in \eqref{pert}. The argument follows closely in \cite{KOPV}.
	
Let $ w $ be the global solution to the energy-critical cubic NLS \eqref{NLS1} with initial data $ w(t_0)=v(t_0)=v_0 $. Then, by assumption $\|w(t_0)\|_{\dot{H}^1}\le R$, and so by \eqref{global-bound}
\[\|w\|_{\dot{X}^1(\R)}\les R.\]
This, together with assumption \textup{(i)}, infer that we can divide the interval $[t_0,T] $ into $ J = J(R ,\theta, \eta) $ many subintervals
$ I_j = [t_j, t_{j+1}] $ so that
\begin{align}\label{divideinte}
\|w\|_{\dot{X}^1(I_j)}+\|f\|_{\dot{X}^1(I_j)}&\le \eta
\end{align}
for some $ \eta\ll \eta_0 $, where $\eta_0$ is dictated by Proposition \ref{PRO:PLWP}. We also write $ [t_0, t_0+\tau] = \bigcup_{j=0}^{J'} ([0, t_0+\tau]\cap I_j) $
for some $ J' \leq J $, where $ [t_0, t_0+\tau] \cap I_j\not=\emptyset $ for $ 0\leq  j  \leq  J' $. 
	
We would like to apply Proposition \ref{LEM:pert} on each interval $I_j$ with $e=g(v,f)$. Starting with $j=0$, we see that \eqref{P2} is automatically satisfied with $E_0=R$ by assumption \textup{(ii)} and \eqref{P3} holds trivially with, say, $E'=1$ since $v(t_0)=w(t_0)$; this also infers that the condition \eqref{P4} holds (for any $\eps$) by the Strichartz estimate. We now turn to \eqref{P1}. Since the nonlinear evolution $ w $ is small on $ I_j $, the linear evolution
$ S(t-t_j) w(t_j) $ is also small on $ I_j $. Indeed, by rearranging the Duhamel formula, we have
\[
S(t-t_j)w(t_j)=w(t)+i\int_{t_j}^{t}S(t-t')\big(|w|^2w\big)(t')dt'
\]
for any $ t\in I_j$; together with the Strichartz, H\"older, Sobolev inequalities, and \eqref{divideinte} we obtain
\begin{align}\label{lin-evol}
\begin{split}
\|S(t-t_j)w(t_j)\|_{\dot{X}^1(I_j)}&\le \|w\|_{\dot{X}^1(I_j)}+C\|w^2\nb w\|_{L^2_tL_x^\frac{4}{3}(I_j\times \R^4)}\\
&\le \eta + C\|\nb w\|_{L^6_tL^{\frac{12}{5}}_x(I_j\times \R^4)}\|w\|^2_{L^6_{t,x}(I_j\times \R^4)}\\
&\le \eta+C\eta^3\\
&\le 2\eta,
\end{split}
\end{align}
since $\eta\ll \eta_0\le 1$. By Proposition \ref{PRO:PLWP} together with \eqref{divideinte} and \eqref{lin-evol} for $j=0$, $v$ exists on the interval $I_0$, moreover,
\begin{align*}
\|v\|_{\dot{X}^1(I_0)}\le \|S(t-t_0)v_0\|_{\dot{X}^1(I_0)}+\|v-S(t-t_0)v_0\|_{\dot{X}^1(I_0)}\le 6\eta
\end{align*}
Thus by the Sobolev embedding $\dot{W}^{1,\frac{12}{5}}(\R^4)\subset L^6(\R^4) $, we have $ \|v\|_{L^6_{t,x}(I_0\times\R^4)}\leq C'\eta $ for some absolute constant $ C'.$ Therefore, \eqref{P1} in Lemma \ref{LEM:pert} is satisfied with $ L=C'\eta .$ 
Let us now verify \eqref{P5}, that is, we need to estimate $\|\nb e\|_{\dot{N}^0(I_0)}=\|\nb g(v,f)\|_{\dot{N}^0(I_0)}$. In view of \eqref{pert}, we distribute the derivative to each term and apply the Strichartz estimate to each contribution, and put the cubic, square and linear terms in $L^2_tL^\frac{4}{3}_x$, $L^\frac{6}{5}_tL^\frac{12}{7}_x$ and $L^1_tL^2_x$ respectively. We then use the H\"older and Sobolev inequalities to put each term in $\dot{X}^1(I)$ (as seen in \eqref{REF:LWP1}). From \eqref{condition} and \eqref{condition2} we obtain
\begin{align}\label{REF:GWP2}
\begin{split}
\|\nb e\|_{\dot{N}^0(I_0)}
	%&\lesssim \bigg(\|v^2\nb v\|_{L^2_tL^{\frac{4}{3}}_x(I_0\times\R^4)}+\|\Psi^2\nb \Psi\|_{L^2_tL^{\frac{4}{3}}_x(I_0\times\R^4)}\bigg)\notag\\
	%&\quad +\bigg( \|v\nb v\|_{L^{\frac{6}{5}}_tL_x^{\frac{12}{7}}(I_0\times\R^4)} +\|\Psi\nb \Psi\|_{L^{\frac{6}{5}}_tL^{\frac{12}{7}}_x(I_0\times\R^4)} \bigg)\notag\\
	%&\quad +\bigg( \|\nb v\|_{L^{1}_tL_x^{2}(I_0\times\R^4)} +\|\nb \Psi\|_{L^{1}_tL^{2}_x(I_0\times\R^4)} \bigg)\notag\\
&\les\|f\|_{\dot{X}^1(I_0)}^3
+ |I_0|^\frac12\Big(\|v \|_{\dot{X}^1(I_0)}^2+\|f \|_{\dot{X}^1(I_0)}^2\Big)\\
&\quad+|I_0|\Big(\|v \|_{L^{\infty}_t\dot{H}^1_x(I_0\times \R^4)}+\|f \|_{L^{\infty}_t\dot{H}^1_x(I_0\times \R^4)}\Big)\\
&\les |I_0|^{3\theta}+|I_0|^{\frac{1}{2}}(\eta^2+|I_0|^{2\theta}) + |I_0|(R+ M )\\
&\leq C(R,M,\eta) \tau^{\theta_0}
\end{split}
\end{align}

\noi
for some $\theta_0=\theta_0(\theta)>0$. Let $ \eps\in(0, \eps_0) $ to be chosen later, where $\eps_0=\eps_0(R, C', \eta)$ is dictated by Lemma \ref{LEM:pert}. We choose $ \tau=\tau(\eps,\theta, R) $ sufficiently small so that
\begin{align*}
\|\nb e\|_{\dot{N}^0(I_0)}\leq \eps.
\end{align*} 
This verifies \eqref{P5}. Therefore, all hypotheses of Lemma \ref{LEM:pert} are satisfied on the interval $ I_0 $, with $L=C'\eta$, $E_0=R$ and $E'=1$. Hence we obtain
\begin{align}\label{P7b}
\|w-v\|_{\dot{S}^1(I_0)}\leq \bar{C}(R, 1, C'\eta)\eps=:C_0(R,\eta)\eps.
\end{align}

Next step, we consider the second interval $ I_1 $. Again, \eqref{P2} is satisfied automatically with $E_0=R$ by assumption. Since the pair $(\infty, 2)$ is admissible, \eqref{P7b} infers that
\[\|w(t_1)-v(t_1)\|_{\dot{H}^1}\le C_0\eps.\]
By choosing $\eps=\eps(R,\eta)$ sufficiently small, \eqref{P3} holds with $E'=C_0 \eps$. Turning to \eqref{P1}, by the Strichartz inequality, \eqref{P7b} and \eqref{lin-evol}, we have
\begin{align}\label{P4c}
\begin{split}
\|S(t-t_1)v(t_1)\|_{\dot{X}^1(I_1)}
&\le	\big\|S(t-t_1)\big[v(t_1)-w(t_1) \big]\big\|_{\dot{X}^1(I_1)}
+\|S(t-t_1)w(t_1)\|_{\dot{X}^1(I_1)}\\
&\le \|w(t_1)-v(t_1)\|_{\dot{H}^1}+2\eta\\
&\leq C_0(R,\eta)\eps+2\eta\\
&\le 3\eta
\end{split}
\end{align}
	%
	%Then, proceeding as in \eqref{v-I_0} and using \eqref{lin-evol}, \eqref{P3c} and \eqref{Psi-smallness},
	%\begin{align*}
	%\|v\|_{\dot{X}^1(I_1)}&\lesssim \|S(t-t_1)v(t_1)\|_{\dot{X}^1(I_1)}+\bigg( \|v\|^3_{\dot{X}^1(I_0)} + \|\Psi\|^3_{\dot{X}^1(I_0)} \bigg) \notag\\
	%&\quad+T^{\frac{1}{2}}\bigg( \|v\|^2_{\dot{X}^1(I_0)}+\|\Psi\|^2_{\dot{X}^1(I_0)} \bigg) + T\bigg( \|v\|_{L^{\infty}_t \dot{H}^1_x}   + \|\Psi\|_{L^{\infty}_t \dot{H}^1_x}\bigg)\notag\\
	%&\lesssim \|S(t-t_1)w(t_1)\|_{\dot{X}^1(I_1)}+\|S(t-t_1)\big[v(t_1)-w(t_1)\big]\|_{\dot{X}^1(I_1)}\\
	%&\quad+\bigg( \|v\|^3_{\dot{X}^1(I_0)} + \|\Psi\|^3_{\dot{X}^1(I_0)} \bigg)+T^{\frac{1}{2}}\bigg( \|v\|^2_{\dot{X}^1(I_0)}+\|\Psi\|^2_{\dot{X}^1(I_0)} \bigg) \\
	%&\quad + T\bigg( \|v\|_{L^{\infty}_t \dot{H}^1_x}   + \|\Psi\|_{L^{\infty}_t \dot{H}^1_x}\bigg)\notag\\
	%&\lesssim \eta + C(C_H,\eta)\eps+\|v\|^3_{\dot{X}^1(I_1)}+T^{\frac{1}{2}}\|v\|^2_{\dot{X}^1(I_1)}+TC^{\frac{1}{2}}_H.
	%\end{align*}
provided 
\begin{align}\label{I1Condition1}
C_0\eps<\eta.
\end{align}
If this holds, then by \eqref{P4c}, $v$ exists on the interval $I_0$, moreover,
\begin{align*}
\|v\|_{\dot{X}^1(I_1)}\le \|S(t-t_1)v(t_1)\|_{\dot{X}^1(I_1)}+\|v-S(t-t_1)v(t_1)\|_{\dot{X}^1(I_1)}\le 8\eta
\end{align*}
By the Sobolev inequality, we see that \eqref{P1} is satisfied with $ L = C\eta $ as before. Now, for \eqref{P4}, by the Strichartz estimate and \eqref{P7b}, we have
\begin{align*}
\|S(t-t_1)(v(t_1)-w(t_1))\|_{\dot{X}^1(I_1)}&\leq \widetilde{C}C_0\eps
\end{align*}
where $\widetilde{C}$ is the absolute constant coming from the Strichartz estimate. Then, \eqref{P4} is satisfied as long as
\begin{align}\label{I1Condition2}
\widetilde{C}C_0\eps<\eps_0(R,1,C'\eta).
\end{align}
Lastly, we argue as in \eqref{REF:GWP2} to obtain
\begin{align*}
\|\nb e\|_{\dot{N}^0(I_1)}\le \eps\le \widetilde{C}C_0\eps,
\end{align*}
without needing to change $\tau = T(\eps, \theta, R) $. Hence, \eqref{P5} is satisfied provided \eqref{I1Condition1} and \eqref{I1Condition2} hold, which can be done by shrinking $\eps=\eps(R,\eta)$ if necessary. Hence Lemma \ref{LEM:pert} infers that
\begin{align*}
\|v-w\|_{\dot{S}^1(I_1)}\leq \bar{C}(R, 1,C'\eta)\widetilde{C}C_0\eps=:{C}_1(R,\eta)\eps.
\end{align*}
We now recursively define $C_j(R,\eta):=\bar{C}(R,1,C'\eta)\widetilde{C}C_{j-1}$ for $1\le j\le J'$. In other words, $C_j(R,\eta)= \bar{C}(R,1, C'\eta)^{j+1}\widetilde{C}^j$. Arguing iteratively, we have
\begin{align*}
\|v-w\|_{\dot{S}^1(I_j)}\leq C_j\eps
\end{align*}
\noi
as long as 
\begin{align}
\label{IjCondition}
\begin{split}
C_{j-1}\eps&< \eta,\\
C_j\eps&<\eps_0(R,1,C'\eta).
\end{split}
\end{align}
Since $C_j$ is increasing in $j$, we just need to ensure that \eqref{IjCondition} holds for $j=J'$. Recalling that $J'\le J=J(R,\eta)$, we see that \eqref{IjCondition} holds for all $j$ provided that $ \eps$ is chosen sufficiently small, depending only on $R$ and $\eta$. In particular, we have constructed a solution $v$ in the entire interval $[t_0,t_0+\tau]$, where $\tau=\tau(R,\eta,\eps)$. This proves the proposition.
\end{proof}

\section{Proof of the main theorem}
\label{SEC:MainTheorem}
We present the proof of Theorem \ref{THM:0} in this section. The first objective is to obtain an a priori bound for the energy of the solution. Armed with this bound as well as tools from the previous sections, we prove global existence by an iterative application of the perturbation lemma (Lemma \ref{LEM:pert}). Finally, we conclude the argument by proving unconditional uniqueness.

\subsection{Bound on the energy}\label{SEC:energy}
Recall the definition of the energy $E(u)(t)$ from \eqref{energy}. Our goal in this subsection is to state and prove a priori bound on the energy. This a priori bound follows
from Ito's lemma and the Burkholder-Davis-Gundy inequality.
In order to justify an application of Ito's lemma, 
one needs to go through a certain approximation argument.
See   [\cite{BD} Proposition 3.2] and [\cite{OO} appendix] for details.

As for a consequance of Proposition \ref{PRO:PLWP} and we follow up from Remark \ref{RM:LWP}. We state the following local well-posedness theory:
 \begin{proposition}
\label{PRO:LWP2}
Let $\phi\in \HS(L^2(\R^4);H^1(\R^4))$. 
Then,
given any $u_0 \in \mathcal E(\R^4)$, there exists an almost surely positive stopping time 
$T = T_{\omega}(u_0)$ and a unique
 local-in-time solution 
$u \in~C([0,T];\mathcal E(\R^4))$
to the energy-critical SNLS \eqref{SNLS}. 
Furthermore, the following blowup alternative holds; 
let $T^{*} = T_{\omega}^{*}(u_0)$ be the forward maximal time of existence. Then, either

\[
T^{*}=\infty 
\quad
\text{or}
\quad
\lim_{T\to T^*} \|u\|_{\dot X^1(T)}=\infty.
\]
\end{proposition}

Then, next proposition provide a priori control on the energy:
\begin{proposition}\label{Pro:Bound-Ham}
Assume $\phi\in \HS(L^2(\R^4);H^1(\R^4))$ and $u_0 \in \mathcal E(\R^4)$. 
Then,  for a given $T>0$:

\smallskip
\noi
\textup{(i)}
for any $t\in [0,T]$, the energy $E(u)(t)$ defined in \eqref{energy} can be expressed as
\begin{align}
E(u)(t)&=E(u_0)
\label{Eq:ham1}
+t\Big(\| \phi\|^2_{\HS(L^2;\dot{H}^1)}+\|\phi\|_{\HS(L^2;L^2)}^2\Big)\\
\label{Eq:ham2}
&\quad +\sum_{n \in \N}\iint_{[0,t]\times\R^4}\Big(|v^*|^2+\Im(\cj{v^*})^2+4\Re(v^*)\Big)|\phi_n|^2dt'dx\\
\label{Eq:ham3}
&\quad +\Im{ \iint_{[0,t]\times\R^4} \Big(|v^*|^2\overline{v^*}-\Dl \cj{v^*}+|v^*|^2+2\Re(v^*) \overline{v^*}+2\Re(v^*)\Big) \phi dWdx }.
\end{align}
\smallskip

\noi
\textup{(ii)} 
Moreover, given $T_0>0$, there exists a constant 
\[C_E= C \big( E(u_0), T_0, \| \phi \|_{\HS(L^2;H^1)}\big)>0\]
such that for any stopping time $T$ with $0<T< \min (T^{\ast}, T_0)$ almost surely, we have
\begin{align}\label{ham-bound}
\E \bigg[ \sup_{0\le t \le T} E(u) (t) \bigg] \le C_E.
\end{align}
where $u$ is the solution to the defocusing energy-critical SNLS \eqref{SNLS} with $u|_{t = 0}=u_0$
and 
$T^{\ast} = T^{\ast}_{\omega} (u_0)$ is the forward maximal time of existence.

\end{proposition}

\begin{proof}
The expression on $E(u)(t)$ follows from a standard application of Ito's Lemma. We turn to prove \eqref{ham-bound}. The term \eqref{Eq:ham1} is easily bounded:
\begin{align}\label{Eq:bound-ha1}
\E\bigg[\sup_{t_0\le t\le T}\eqref{Eq:ham1}\bigg]
&\les T\|\phi\|_{\HS(L^2;{H}^1)}.
\end{align}

\noi
Turning our attention to \eqref{Eq:ham2}, by the H\"older, Sobolev and Young inequalities, we have
\begin{align}\label{Eq:bound-ha2}
\begin{split}
\E\bigg[\sup_{0\le t \le T} \eqref{Eq:ham2}\bigg]
&\le C\E\bigg[\sum_{n \in \N}  \int_{[0,T]}   \Big(\|v^*\|_{L^4}^2+\| |v^*|^2 + 2 \Re(v^*)\|_{L^2}\Big) \|\phi_n\|^2_{L^4} \,dt'\bigg]\\
&\leq  C T\|\phi\|_{\HS(L^2;H^1)}^2\E\bigg[\sup_{0\le t \le T} \Big(E(u)(t)\Big)^\frac12\bigg]\\
&\leq CT^2 \|\phi \|_{\HS(L^2;\dot{H}^1)}^4+\frac{1}{8} \E\bigg[\sup_{0\le t \le T}E(u)(t)\bigg].
\end{split}
\end{align}

\noi
Finally, we bound \eqref{Eq:ham3}. By the Burkholder-Gundy-Davis, H\"older, Sobolev and Young inequalities, we have 
\begin{align*}
\begin{split}
\E\bigg[  \sup_{0 \le t \le T}  &\Im{ \iint_{[0,t]\times\R^4} |v^*|^2\cj{v^*} \phi dWdx }  \bigg]\\
&\le C \E\bigg[\bigg( \sum_{n \in \N} \int_{0}^{T} \bigg|\int_{\R^4} |v^*|^2\cj{v^*} \phi_n  dx \bigg|^2 dt'  \bigg)^{\frac{1}{2}} \bigg]\\
&\le C \E\bigg[  \bigg( \sum_{n \in \N} \int_{0}^{T}  \big\| |v^*|^2v^* \big\|_{\dot{H}^{-1}}^2  \|\phi_n\|_{\dot{H}^1}^2 \,dt'  \bigg)^{\frac{1}{2}} \bigg]\\
&\le C\E\bigg[  \bigg( \sum_{n \in \N} \int_{0}^{T} \big\|v^* \big\|_{L^{4}}^6  \|\phi_n\|_{\dot{H}^1}^2 dt'  \bigg)^{\frac{1}{2}} \bigg]\\
&\le C \|\phi\|_{\HS(L^2;H^1)}\E\bigg[\sup_{0\le t\le T}E(u)(t)^\frac38\bigg]\\
&\le C \|\phi\|_{\HS(L^2;H^1)}\E\bigg[1+\sup_{0\le t\le T}E(u)(t)^\frac12\bigg]\\
&\le C \|\phi\|_{\HS(L^2;H^1)}+C \|\phi\|_{\HS(L^2;H^1)}^2+\frac1{32}\E\bigg[\sup_{0\le t\le T}E(u)(t)\bigg].
\end{split}
\end{align*}

\noi
where we used the elementary fact $A^\frac38\le \min(1, A^\frac12)\le 1+A^\frac12$ in the penultimate inequality. The rest of contributions from \eqref{Eq:ham3} are controlled in a similar manner and we omit the details. Ultimately, we obtain
\begin{align}\label{Eq:bound-ha3}
\E\bigg[\sup_{0\le t\le T}\eqref{Eq:ham3}\bigg]\le C(\phi)+\frac18 \E\bigg[\sup_{0\le t\le T}E(u)(t)\bigg].
\end{align}
Combining \eqref{Eq:bound-ha1}--\eqref{Eq:bound-ha3} concludes the proof.
\end{proof}

\subsection{Global existence of SNLS}\label{SUB:global}
We are now ready to finish off the proof of the existence part of Theorem \ref{THM:0}. 
Recall that given a local-in-time solution $ u $ to \eqref{SNLS}, let $ v=u-1-\Psi $. Then, $ v $ satisfies \eqref{SNLS2}. The global existence part of Theorem \ref{THM:0} follows from applying Proposition \ref{Pro:perturb2} to equation \eqref{PNLS} with $ f=\Psi $, once we verify the hypotheses \textup{(i)} and \textup{(ii)}.

Let $T>0$.
From Proposition \ref{Pro:Bound-Ham} and Markov's inequality, we have the following almost sure a priori bound:
\begin{align}\label{a.s bund}
\sup_{0\leq t \leq T} E(u)(t)\leq C(\omega, T, E(u_0), \|\phi\|_{\HS(L^2;H^1)})< \infty.
\end{align}

\noi
for a solution $ u=v+1+\Psi $ to \eqref{SNLS}.
Since we set $v=u-1-\Psi=v^*-\Psi$, according to the definition of $E(u)$ from \eqref{energy}, and \eqref{a.s bund} we obtain
\begin{align*}
\|v\|_{L^\infty_t\dot{H}_x^1([0,T]\times\R^4)}
&\le  \|v^*\|_{L^\infty_t\dot{H}_x^1([0,T]\times\R^4)} + \|\Psi\|_{L^\infty_t\dot{H}^1_x([0,T]\times\R^4)}\\
&\le \sup_{0\le t\le T}\big[E(u)(t)]^\frac12 + \|\Psi\|_{L^\infty_t\dot{H}^1_x([0,T]\times\R^4)}\\
&\leq C(\omega, T, E(u_0), \|\phi\|_{\HS(L^2;H^1)})< \infty
\end{align*}

\noi
almost surely.
This shows that hypothesis \textup{(ii)} in Proposition \ref{Pro:perturb2} holds almost surely for some almost sure finite $ R=R(\omega)\geq 1 $.
On the other hand, 
we can verify hypothesis \textup{(i)} in Proposition \ref{Pro:perturb2}
by using
H\"older's inequality in time,
Markov's inequality,
and Lemma \ref{LEM:sccont}.
More precisely, by fixing finite $q > 3$ and noting $\frac{12}{5} < 4$ for $ d = 4$.
Then, 
 Lemma \ref{LEM:sccont} (ii) yields
\begin{align*}
\E \bigg[  \| \nb \Psi \|_{L^q([0,T];  L^{\frac{12}{5}}(\R^4))} \bigg] \le C \|\nb \phi \|_{\HS(L^2;L^2)}.
\end{align*}

\noi
Then, 
Markov's inequality yields
\begin{align}
\label{p1}
 \| \nb \Psi \|_{L^q([0,T];  L^{\frac{12}{5}}(\R^4))}  \leq C(\phi,\omega)<\infty.
\end{align}

\noi
which in turn implies $\Psi \in \dot{X}^1(I)$ almost surely (when $q=6>3$).
Moreover, it follows from \eqref{p1}
and H\"older's inequality in time that
\begin{align}
\label{Psi-Bound}
%\|\Psi\|_{L^\infty_t\dot{H}^1_x(I\times \R^4)}+
\|\Psi\|_{\dot{X}^1(I)}&\le 
|I|^{\theta} \| \nb \Psi \|_{L^q([0,T];  L^{\frac{12}{5}}(\R^4))} 
\leq
C(\phi,\omega)|I|^\theta
% \quad\quad\quad\mbox{ for any } I\subseteq [0,T],
\end{align}

\noi
{for any } $I\subseteq [0,T]$, where $\theta = \frac16-\frac{1}{q}>0$.
This verify \eqref{condition}.
Furthermore,
let $\phi\in \HS(L^2(\R^4);H^1(\R^4))$.
Lemma \ref{LEM:sccont} (i) 
guarantees that $\Psi \in C([0,T]; H^1(\R^4))$. Then, same as in \eqref{p1}, by Markov's inequality yields
\begin{align*}
%\label{p2}
 \|  \Psi \|_{L^\infty([0,T];  H^1(\R^4))}  \leq C(\phi,\omega)<\infty.
\end{align*}

\noi
This shows that hypothesis \textup{(i)} in Proposition \ref{Pro:perturb2} holds almost surely.

Hence, we can invoke Proposition \ref{Pro:perturb2} to extend the solution $v$ to \eqref{SNLS2} on $[0,T]$. From the discussion in the introduction, we reduce equation \eqref{SNLS} into \eqref{SNLS2} by \emph{Da Prato-Debussche trick}. Hence, it is enough to complete the existence part of the proof. 

\subsection{Unconditional uniqueness}
We turn now to showing that the global solutions
constructed above are unique among those that are continuous (in time) with values
in the energy space. We mimic the arguments in \cite{CKSTT} and \cite{KOPV}. To this end, let $ v_0 $
be such that $ 1 + v_0 \in \cE(\R^4) $ and let $ v $ be the global solution to \eqref{SNLS2} constructed in Subsection \ref{SUB:global}. In particular, $ v \in \dot{S}^1(I) $ for any compact time interval $ I $.
Let $ \wt v : [0, t']\times \R^4 \rightarrow \C $ be a second solution to \eqref{SNLS2} with the same initial data such
that $ 1+\wt v\in C([0, t']; \cE(\R^4)) $ almost surely and write $ z := (1+v)-(1+\wt v)= v-\wt v $. In what follows, we fix an $\omega\in\Omega$ for which both $v$ and $\widetilde{v}\in C([0,t'];\cE(\R^4))$. As $ z(0) = 0 $ and $ z$ is continuous in time, shrinking $ t' $ if necessary, we may assume
\begin{align}\label{un-con1}
\|\Re z\|_{L^{\infty}_tH^1_x([0,t']\times\R^4)} + 
\|\Im z\|_{L^{\infty}_t\dot{H}^1_x([0,t']\times\R^4)}\leq \eta
\end{align}
for a small $ \eta> 0 $ to be chosen shortly. By the Sobolev embedding $ \dot{H}^1(\R^4)\subset L^4(\R^4) $, this yields
\begin{align}\label{un-con2}
\|z\|_{L^{\infty}_t L^4_x([0,t']\times\R^4) }\lesssim\eta,
\end{align}
in particular, we have $ z \in L^2_tL^4_x([0, t'] \times \R^4) $. Recalling that $ v\in\dot{S}^1(I) $ almost surely for any compact time
interval $ I $ and further shrinking $ t' $ if necessary, we may also assume (by Sobolev embedding $ W^{1,\frac{12}{5}}(\R^4)\subset L^6(\R^4) $) that
\begin{align}\label{un-con3}
\|v\|_{L^6_{t,x}([0,t']\times\R^4)}\leq\eta.
\end{align}
On the other hand, as seen in the previous subsection, one can find an event of arbitrarily large probability such that \eqref{Psi-Bound} holds. Hence we may assume that $\omega$ lies in this event. By Sobolev embeddings $\dot{H}^1(\R^4)\subset L^4(\R^d)$ and $\dot{W}^{1,\frac{12}{5}}\subset L^6(\R^4)$, as well as shrinking $t'$ if necessary, we have
\begin{align}\label{un-con2b}
\|\Psi\|_{L^{\infty}_t L^4_x([0,t']\times\R^4) }&\le \eta,\\
\label{un-con3b}
\|\Psi\|_{L^6_{t,x}([0,t']\times\R^4)}&\leq\eta.
\end{align}

\noi
Now, we consider the following difference
\begin{align*}
\bigg| \big[ |v|^2v&+g(v,\Psi)  \big] - \big[  |\wt v|^2\wt v+g(\wt v,\Psi)       \big] \bigg|\\
&\les \big(   \Re(\Psi)|z|+|\Re(z)|+|\Psi||\Re(z)|+|\Psi|^2|z| \\
&\quad +|\Re(\Psi)||\Psi||\Re(z)|+|\Re(\Psi)||\Re(\overline{z})|+|z|^2\\
&\quad +|z||v|+|\Psi||z|^2  +|\Psi||z||v|+|\Re(\Psi)||z|^2\\
&\quad+|\Re(\Psi)||z||v|+|z|^3+|z||v|^2  \big)\\
&\les \mathcal O\big( |z|^3+|zv|^2+|\Psi z^2|+|z|^2 +|\Psi zv| +|\Psi z| +|\Psi^2z|+|zv| +|\Re(z)| \big).
\end{align*}

\noi
By the Strichartz and H\"older inequalities together with \eqref{un-con1}-\eqref{un-con3b}, we have
\begin{align*}
\|z&\|_{L^2_tL^4_x}+\|\Re(z)\|_{L^{\infty}_tL^2_x}\\&\lesssim \|z^3\|_{L^2_tL^{\frac{4}{3}}_x}+\|z v^2\|_{L^{\frac{6}{5}}_tL_x^{\frac{12}{7}}}+\|z^2\|_{L^1_tL^2_x}+\|z v\|_{L^1_tL^2_x} + \|\Psi z^2\|_{L^2_tL^{\frac{4}{3}}_x}\\
&\quad+\|\Psi z v\|_{L^{\frac{6}{5}}_tL_x^{\frac{12}{7}}}+\|\Psi z\|_{L^1_tL^2_x}+\|\Psi^2 z \|_{L^2_tL^\frac{4}{3}_x}+\|\Re(z)\|_{L^1_tL^2_x}\\
&\lesssim \|z\|_{L^2_tL^4_x}\big[ \|z\|^2_{L^{\infty}_tL^4_x} + \|v\|^2_{L^6_{t,x}}+ t'^{\frac{1}{2}}\|z\|_{L^{\infty}_tL^4_x} \\
& \quad + t'^{\frac12}\|v\|_{L^{\infty}_tL^4_x}
+\|\Psi\|_{L^{\infty}_tL^4_x} \|z\|_{L^{\infty}_tL^4_x}+\|\Psi\|_{L^6_{t,x}} \|v\|_{L^6_{t,x}}\\
&\quad + t'^{\frac{1}{2}}\|\Psi\|_{L^{\infty}_tL^4_x}+\|\Psi\|^2_{L^6_{t,x}}
\big]+t'\|\Re(z)\|_{L^{\infty}_tL^2_x}\\
&\lesssim  \|z\|_{L^2_tL^4_x}(\eta^2+\eta t'^{\frac12}+t'^{\frac12})+t'\|\Re(z)\|_{L^{\infty}_tL^2_x}.
\end{align*}
where we omitted the domain $ [0, t']\times\R^4 $ above for the sake of readability. Taking $ \eta $ sufficiently
small and shrinking $ t' $ further if necessary, we obtain
\begin{align*}
\|z\|_{L^2_tL^4_x([0,t']\times \R^4)} + 
\|\Re(z)\|_{L^{\infty}_t L^2_x([0,t' ]\times \R^4)} = 0,
\end{align*}
which proves $ v =\wt v $ almost surely on $ [0, t'] \times \R^4 $.

By time translation invariance, this argument can be applied to any sufficiently short time interval, which yields global unconditional uniqueness. This completes the proof of Theorem \ref{THM:0}.

\begin{ack}\rm 
	
The authors would like to thank their advisors Tadahiro Oh and Oana Pocovniu for proposing this problem and their continuing support. The authors are also grateful to Mamoru Okamoto, Yuzhao Wang, and Justin Forlano for several helpful discussions on the present paper. 
The authors especially would like to thank the anonymous referee for helpful comments.
K.C.~and G.L.~were supported by The Maxwell Institute Graduate School in Analysis and its Applications, a Centre for Doctoral Training funded by the UK Engineering and Physical Sciences Research Council (Grant EP/L016508/01), the Scottish Funding Council, Heriot-Watt University and the University of Edinburgh. 
K.C.~was partly
supported by the Tadahiro Oh's European Research Council (grant no. 637995 ``ProbDynDispEq").
G.L.~also acknowledges support
from Tadahiro Oh's European Research Council (grant no. 637995 ``ProbDynDispEq").

\end{ack}

\end{document}